\documentclass[11pt]{article}
\usepackage{enumerate}
\usepackage{amssymb,amsmath,amsfonts,amsthm,mathrsfs}
\usepackage[colorlinks=true, pdfstartview=FitV, linkcolor=blue, citecolor=blue, urlcolor=blue]{hyperref}
\usepackage{graphicx,color}
\usepackage{cite}
\usepackage{indentfirst}

\allowdisplaybreaks

\begin{document}
\def\rn{{\mathbb R^n}}  \def\sn{{\mathbb S^{n-1}}}
\def\co{{\mathcal C_\Omega}}
\def\z{{\mathbb Z}}
\def\nm{{\mathbb (\rn)^m}}
\def\mm{{\mathbb (\rn)^{m+1}}}
\def\n{{\mathbb N}}
\def\cc{{\mathbb C}}

\newtheorem{defn}{Definition}
\newtheorem{thm}{Theorem}
\newtheorem{lem}{Lemma}
\newtheorem{cor}{Corollary}
\newtheorem{rem}{Remark}

\title{\bf\Large Sharp bounds for a class of integral operators in weighted-type spaces on Heisenberg group
\footnotetext{
\newline\indent\hspace{1mm} {\it 2020 Mathematics Subject Classification}:  Primary 42B35; Secondary 26D15.}}

\date{}
\author{Xiang Li, Zhanpeng Gu,Dunyan Yan and Zhongci Hang\footnote{Corresponding author}}
\maketitle

\begin{center}
\begin{minipage}{13cm}
{\small {\bf Abstract:}\quad
In this paper, we will use the conclusions and methods in \cite{1} to obtain the sharp bounds for a class of integral operators with the nonnegative kernels in weighted-type spaces on Heisenberg group. As promotions, the sharp bounds of Hardy operator , Hardy Littlewood-P\'{o}lya operator and Hilbert operator are also obtained.}
\end{minipage}
\end{center}

\begin{center}
\section{Introduction}\label{sec1}
\end{center}

 In \cite{2}, Stein and Weiss gave a class of integral operators with the nonnegative kernels which satisfy some homogeneity and the rotational invariance. Beckner \cite{3} represented it as Stein-Weiss lemma and obtained the following theorem.
\begin{thm}\label{mian_01}
	Suppose $K$ is a non-negative kernel defined on $\mathbb{R}^n \times \mathbb{R}^n$ that satisfies continuous on any domain that excludes the point $(0,0)$, homogeneous of degree $-n, K(\delta u, \delta v)=\delta^{-n} K(u, v)$, and $K(R u, R v)=$ $K(u, v)$ for any $R \in S O(n)$. Then $T$ is a from $L^p(\mathbb{R}^n)$ to $L^p(\mathbb{R}^n)$ bounded integral operator defined by
	$$
	T f(x)=\int_{\mathbb{R}^n} K(x, y) f(y) d y
	$$
	with
	\begin{equation}\label{main_1}
		\|T f\|_{L^p\left(\mathbb{R}^n\right)} \leq A\|f\|_{L^p\left(\mathbb{R}^n\right)}
	\end{equation}
	holds for $1<p<\infty$, the optimal constant is given by
	\begin{equation}\label{main_2}
		A=\int_{\mathbb{R}^n} K\left(x, \hat{e}_1\right)|x|^{-n / p^{\prime}} d x
	\end{equation}
	and $\hat{e}_1$ is a unit vector in the first coordinate direction.
\end{thm}

The inequality (\ref{main_1}) was given by Stein and Weiss \cite{2}. In \cite{3} and \cite{4}, Beckner proved that the constant $A$ in (\ref{main_2}) is the best constant.
Inspired by Theorem \ref{mian_01}, we investigate a class of integral operators on Heisenberg group and obtained sharp bounds in weighted type spaces.

Let us first recall some basic knowledge about the Heisenberg group. The Heisenberg group $\mathbb{H}^n$ is a non-commutative nilpotent Lie group, with the underlying manifold $\mathbb{R}^{2n}\times\mathbb{R}$ with the group law
$$
x \circ y=\left(x_1+y_1, \ldots, x_{2 n}+y_{2 n}, x_{2 n+1}+y_{2 n+1}+2 \sum_{j=1}^n\left(y_j x_{n+j}-x_j y_{n+j}\right)\right)
$$
and
$$
\delta_r\left(x_1, x_2, \ldots, x_{2 n}, x_{2 n+1}\right)=\left(r x_1, r x_2, \ldots, r x_{2 n}, r^2 x_{2 n+1}\right), \quad r>0,
$$
where $x=(x_1,\cdots,x_{2n},x_{2n+1})$ and $y=(y_1,\cdots,y_{2n},y_{2n+1})$.
The Haar measure on $\mathbb{H}^n$ coincides with the usual Lebesgue measure on $\mathbb{R}^{2n+1}$. We denote the measure of any measurable set $E \subset \mathbb{H}^n$ by $|E|$.Then
$$
|\delta_r(E)|=r^Q|E|, d(\delta_r x)=r_Q d x,
$$
where $Q=2n+2$ is called the homogeneous dimension of $\mathbb{H}^n$.

The Heisenberg distance derived from the norm
$$
|x|_h=\left[\left(\sum_{i=1}^{2 n} x_i^2\right)^2+x_{2 n+1}^2\right]^{1 / 4},
$$
where $x=(x_1,x_2,\cdots,x_{2n},x_{2n+1})$,is given by
$$
d(p, q)=d(q^{-1} p, 0)=|q^{-1} p|_h.
$$
This distance $d$ is left-invariant in the sense that $d(p,q)$ remains unchanged when $p$ and $q$ are both left-translated by some fixed vector on $\mathbb{H}^n$.Furthermore,$d$ satisfies the triangular inequality\cite{6}
$$
d(p, q) \leq d(p, x)+d(x, q), \quad p, x, q \in \mathbb{H}^n.
$$
For $r>0$ and $x\in\mathbb{H}^n$,the ball and sphere with center $x$ and radius $r$ on $\mathbb{H}^n$ are given by
$$
B(x, r)=\{y \in \mathbb{H}^n: d(x, y)<r\},
$$
and
$$
S(x, r)=\{y \in \mathbb{H}^n: d(x, y)=r\},
$$
respectively.And we have
$$
|B(x, r)|=|B(0, r)|=\Omega_Q r^Q,
$$
where
$$
\Omega_Q=\frac{2 \pi^{n+\frac{1}{2}} \Gamma(n / 2)}{(n+1) \Gamma(n) \Gamma((n+1) / 2)}
$$
is the volume of the unit ball $B(0,1)$ on $\mathbb{H}^n$, $\omega_Q=Q \Omega_Q$ (see \cite{5}). For more details about Heisenberg group can be refer to \cite{7} and \cite{8}.

The weighted-type space on Euclidean space $H^\infty_\alpha(\mathbb{R}^n)$ with $\alpha>0$ consists all measurable functions $f$ satisfying
$$
\|f\|_{H^\infty_\alpha(\mathbb{R}^n)}:=\text{ess}\sup_{x\in\mathbb{R}^n}|x|^\alpha|f(x)|<\infty.
$$
The function $\|\cdot\|_{H^\infty_\alpha(\mathbb{R}^n)}$ is a norm on the space which is $dx$ almost everywhere equal.
In this paper, we will consider the weighted-type space on Heisenberg group. Our definition is as follow.
\begin{defn}
	The weighted-type space on Heisenberg group  $H^\infty_\alpha(\mathbb{H}^n)$ with $\alpha>0$ consists of all measurable functions $f$ satisfying
	$$
	\|f\|_{H^\infty_\alpha(\mathbb{H}^n)}:=\text{ess}\sup_{x\in\mathbb{H}^n}|x|_h^\alpha|f(x)|<\infty,
	$$
	where $\|\cdot\|_{H_\alpha^\infty}$ is a norm on the space which also is $dx$ almost everywhere equal.
\end{defn}
In harmonic analysis, the study of multilinear integral operators and their properties is often more desirable and interesting. In \cite{9}, Wu and Yan extened Theorem \ref{mian_01} to the case of multilinearity, now we utilize their promotional methods to give $m$-linear integral operator with the nonnegative kernels on Heisenberg group.
\begin{defn}
	Suppose that $K$ is a nonnegative kernel defined on $\mathbb{H}^{nm}$ and satisfies the homogenous degree $-mn$,
	$$
	K(|t|_h x, |t|_h y_1,\ldots,|t|_h y_m)=|t|_h^{-mn}K(x,y_1,\ldots,y_m),
	$$
	and
	$$
	K(R x,R y_1,\ldots,R y_m)=K(x,y_1,\ldots,y_m)
	$$
	for any $R\in SO(n)$, then $K$ denotes the kernel of
	\begin{equation}\label{main_3}
		T_m(f_1,\ldots,f_m)(x)=\int_{\mathbb{H}^{nm}}K(x,y_1,\ldots,y_m)f_1(y_1)\cdots f_m(y_m) dy_1\cdots dy_m.
	\end{equation}
	Taking $e_1$ is a unit vector in the first coordinate direction in $\mathbb{H}^n$, then we have
	\begin{equation}\label{main_4}
		T_m(f_1,\ldots,f_m)(x)=\int_{\mathbb{H}^{nm}}K(e_1,y_1,\ldots,y_m)f_1(\delta_{|x|_h} y_1)\cdots f_m(\delta_{|x|_h}y_m) dy_1\cdots dy_m.
	\end{equation}
\end{defn}
Next, we will provide our main result of mulitilinear integral operator with kernel $K$ in weighted type space on Heisenberg group.

\begin{center}
\section{Sharp bound for m-linear integral operator}\label{sec2}
\end{center}
\par

\begin{thm}\label{main_02}
	Let $\alpha>0$ and $\alpha=\alpha_1+\cdots+\alpha_m$ with $\alpha_j>0$ for any $j=1,\ldots,m$, if
	\begin{equation}\label{main_5}
		H_m:=\int_{\mathbb{H}^{nm}}K(e_1,y_1,\ldots,y_m)\prod_{i=1}^m|y_i|_h^{-\alpha_i}dy_1,\ldots d y_m<\infty,
	\end{equation}
	then $T(f_1,\ldots f_m)(x)$ is bounded from $H_{a_1}^\infty(\mathbb{H}^n)\times\cdots\times H_{\alpha_m}^\infty(\mathbb{H}^n)$ to $H_\alpha^\infty(\mathbb{H}^n)$.
	
	Moreover, if condition (\ref{main_4}) is satisfied, then the following formula for the norm of the operator hold
	\begin{equation}\label{main_12}
		\|T_m\|_{\prod_{i=1}^m H^\infty_{\alpha_i}(\mathbb{H}^n)\rightarrow H^\infty_\alpha(\mathbb{H}^n)}=H_m.
	\end{equation}
\end{thm}
\begin{proof}[Proof of Theorem \ref{main_02}]
	For convenience, we take
	$$
	T_m(f_1,\ldots,f_m)(x)=\int_{\mathbb{H}^{nm}}K(e_1,y_1,\ldots,y_m)\prod_{i=1}^m f_i(\delta_{|x|_h}y_j)dy_1\cdots dy_m.
	$$
	Using the above identity and the definition of weighted-type space on Heisenberg group $H^\infty_\alpha(\mathbb{H}^n)$, we can infer that
	$$
	\begin{aligned}
		&\|T_m\|_{H^\infty_\alpha(\mathbb{H}^n)}\\
		=&\text{ess}\sup_{x\in\mathbb{H}^n}|x|_h^\alpha\left|\int_{\mathbb{H}^{nm}}K(e_1,y_1,\ldots,y_m)\prod_{i=1}^m f_i(\delta_{|x|_h}y_i)d y_1\cdots dy_m\right|\\
		\leq&\prod_{i=1}^m\|f_i\|_{H^\infty_{\alpha_i}(\mathbb{H}^n)}\sup_{x\in\mathbb{H}^n\backslash\{0\}}|x|_h^\alpha\left|\int_{\mathbb{H}^{nm}}K(e_1,y_1,\ldots,y_m)\prod_{i=1}^m|\delta_{|x|_h}y_i|_h^{-\alpha_i}dy_1\cdots y_m\right|\\
		=&\prod_{i=1}^m\|f_i\|_{H^\infty_{\alpha_i}(\mathbb{H}^n)}\sup_{x\in\mathbb{H}^n\backslash\{0\}}|x|_h^\alpha\left|\int_{\mathbb{H}^{nm}}K(e_1,y_1,\ldots,y_m)\prod_{i=1}^m|x|_h^{-\alpha_i} |y_i|_h^{-\alpha_i}dy_1\cdots y_m\right|\\
		=&H_m\prod_{i=1}^m\|f_i\|_{H^\infty_{\alpha_i}(\mathbb{H}^n)}
	\end{aligned}
	$$
	for every $(f_1,\ldots,f_m)\in \prod_{i=1}^m H^\infty_{\alpha_i}(\mathbb{H}^n)$, from which by taking the supremum in $H_{\alpha_i}^{\infty}(\mathbb{H}^n)$, it follows that
	\begin{equation}\label{main_11}
		\|T_m\|_{\prod_{i=1}^mH^\infty_{\alpha_i}(\mathbb{H}^n)\rightarrow H^\infty_{\alpha}(\mathbb{H}^n)}\leq H_m,
	\end{equation}
	and we obtain the boundedness of $T_m$ on $H^\infty_\alpha(\mathbb{H}^n)$.
	
	Then by taking
	\begin{equation}\label{main_6}
		f_{\alpha_{j}}(x)=
		\begin{cases}
			1/|x|^{\alpha_{j}}_{h},       &       {x\neq 0},\\
			0,  &   {x=0},
		\end{cases}
	\end{equation}
	it is clear that for any $j=1,\ldots,y_m$, we have
	\begin{equation}\label{main_8}
		\|f_{\alpha_j}\|_{H^\infty_{\alpha_j}(\mathbb{H}^n)}=1.
	\end{equation}
	Bring (\ref{main_6}) into identity (\ref{main_4}) and through simple calculations, we have
	\begin{equation}\label{main_7}
		\begin{aligned}
			&|x|_h^\alpha\left|\int_{\mathbb{H}^{nm}}K(e_1,y_1,\ldots,y_m)\prod_{j=1}^mf_{\alpha_j}(\delta_{|x|_h}y_j)dy_1\cdots dy_m\right|\\
			=&|x|_h^\alpha\left|\int_{\mathbb{H}^{nm}}K(e_1,y_1,\ldots,y_m)\prod_{j=1}^m|\delta_{|x|_h}y_j|_h^{-\alpha_j}dy_1\cdots dy_m\right|\\
			=&\int_{\mathbb{H}^{nm}}K(e_1,y_1,\ldots,y_m)\prod_{j=1}^m|y_j|_h^{-\alpha_j}dy_1\cdots y_m=H_m
		\end{aligned}
	\end{equation}
	for $x\neq 0$, where we used the condition $\alpha=\alpha_1+\cdots+\alpha_m$.
	From (\ref{main_7}) and by using the definition of norm in the weighted-type space on Heisenberg group $H^\infty_\alpha(\mathbb{H}^n)$, we can infer that
	\begin{equation}\label{main_9}
		\begin{aligned}
			&\|T_m\|_{H^\infty_\alpha(\mathbb{H}^n)}\\
			=&\text{ess}\sup_{x\in\mathbb{H}^n}|x|_h^\alpha\left|\int_{\mathbb{H}^{nm}}K(e_1,y_1,\ldots,y_m)\prod_{j=1}^m f_{\alpha_j}(\delta_{\alpha_j}y_j)dy_1\cdots dy_m\right|=H_m.
		\end{aligned}
	\end{equation}
	Combining (\ref{main_8}) with (\ref{main_9}) it follows that
	\begin{equation}\label{main_10}
		\|T_m\|_{\prod_{j=1}^m H^\infty_{\alpha_j}(\mathbb{H}^n)\rightarrow H^\infty_{\alpha}(\mathbb{H}^n)}\geq H_m.
	\end{equation}
	Through (\ref{main_11}) and (\ref{main_10}), it  implies that (\ref{main_12}) holds. We have finished the proof of Theorem \ref{main_02}.
\end{proof}

\vspace{0.1cm}

\begin{center}
\section{Sharp bound for  Hardy operator}\label{sec3}
\end{center}

In this section, we will study $m$-linear $n$-dimensional Hardy operator on Heisenberg group and obtain its sharp bound in weighted-type space on Heisenberg group.

The $n$-dimensional Hardy operator on Heisenberg group is defined by Wu and Fu \cite{10}
\begin{equation}\label{main_331}
	Hf(x):=\frac{1}{|B(0,|x|_h)|}\int_{B(0,|x|_h)}f(y)dy, x\in\mathbb{H}^n\backslash\{0\},
\end{equation}
where $f$ be a locally integrable function on $\mathbb{H}^n$. And they proved $H$ is bounded from $L^p(\mathbb{H}^n)$ to $L^p(\mathbb{H}^n)$, $1<p\leq\infty$.
Moreover,
$$
\|H\|_{L^p(\mathbb{H}^n)\rightarrow L^p(\mathbb{H}^n)}=\frac{p}{p-1},1<p\leq\infty.
$$
In \cite{11}, Fu and Wu introduced the $m$-linear $n$-dimensional Hardy operator, which is defined by
$$
\mathfrak{H}_m(f_1,\ldots,f_m)(x)=\frac{1}{\Omega_n^m|x|^{mn}}\int_{|(y_1,\ldots,y_m)|<|x|}f_1(y_1)\cdots f_m(y_m),
$$
where $x\in\mathbb{R}^n\backslash\{0\}$ and $f_1,\ldots,f_m$ are nonnegative locally integrable functions on $\mathbb{R}^n$.
Inspired by \cite{11}, we  define the m-liner n-dimensional Hardy operator on Heisenberg group is as follow.
\begin{defn}
	Let $m$ be a positive integer and $f_1,\ldots,f_m$ be nonnegative locally integrable functions on $\mathbb{H}^n$, then
	\begin{equation}\label{main_21}
		\mathfrak{H}_m^h(f_1,\ldots,f_m)(x)=\frac{1}{\Omega_Q^m|x|_h^{mQ}}\int_{|(y_1,\ldots,y_m)|_h<|x|_h}f_1(y_1)\cdots f_m(y_m)dy_1\cdots dy_m,
	\end{equation}
	where $x\in\mathbb{H}^n\backslash\{0\}$.
\end{defn}
Now, we formulate our main results as follows.
\begin{thm}\label{mian_021}
	Let $m\in \mathbb{N}$, $f_i\in H^\infty_{\alpha_i}(\mathbb{H}^n)$, $\alpha>0$ and $\alpha=\alpha_1+\cdots+\alpha_m $ with $\alpha_j>0$ for any $j=1,\ldots,m.$ If
	\begin{equation}\label{main_23}
		\mathcal{A}^h_m=\frac{1}{\Omega_Q^m}\int_{|(y_1,\ldots,y_m)|_h<1}\prod_{i=1}^m|y_i|_h^{-\alpha_i}dy_1\cdots d y_m<\infty,
	\end{equation}
	then $\mathfrak{H}_m^h$ is bounded from $H_{a_1}^\infty(\mathbb{H}^n)\times\cdots\times H_{\alpha_m}^\infty(\mathbb{H}^n)$ to $H_\alpha^\infty(\mathbb{H}^n)$, and the following formula for the norm of the operator hold
	$$
	\|\mathfrak{H}_m^h\|_{\prod_{i=1}^m H^\infty_{\alpha_i}(\mathbb{H}^n)\rightarrow H^\infty_\alpha(\mathbb{H}^n)}=\mathcal{A}^h_m.
	$$
\end{thm}
According to Theorem \ref{mian_021}, we have the following result holds.
\begin{cor}\label{main_cor1}
	Let $m\in\mathbb{N}, \alpha>0$ and $\alpha=\alpha_1+\cdots+\alpha_m $ with $\alpha_i>0$ for any $i=1,\ldots,m$.Then the operator $\mathfrak{H}_m^h$ is bounded from $H_{a_1}^\infty(\mathbb{H}^n)\times\cdots\times H_{\alpha_m}^\infty(\mathbb{H}^n)$ to $H_\alpha^\infty(\mathbb{H}^n)$ if and only if (\ref{mian_021}) hold. Moreover, if (\ref{main_23}) hold, the following formula holds
	$$
	\mathcal{A}^h_m=\frac{2Q^m}{2^m(mQ-\alpha)}\frac{\prod_{i=1}^m\Gamma(\frac{Q}{2}-\frac{\alpha_i}{2})}{\Gamma(\frac{mQ}{2}-\frac{\alpha}{2})}.
	$$
\end{cor}
Now, we begin to prove our result.
\begin{proof}[Proof of Theorem \ref{mian_021}]
	The proof process in this section is similar to that of Theorem \ref{main_02}, so we only provide a brief proof and focus on the calculation of the optimal constant.
	
	Combining (\ref{main_21}) with the definition of weighted-type space on Heisenberg group $H^\infty_\alpha(\mathbb{H}^n)$, we have
	$$
	\begin{aligned}
		\|\mathfrak{H}_m^h\|_{H^\infty_\alpha(\mathbb{H}^n)}&=\text{ess}\sup_{x\in\mathbb{H}^n}|x|_h^\alpha\left|\frac{1}{\Omega_Q^m|x|_h^{mQ}}\int_{|(y_1,\ldots,y_m)|_h<|x|_h}f_1(y_1)\cdots f_m(y_m)dy_1\cdots dy_m\right|\\
		&=\text{ess}\sup_{x\in\mathbb{H}^n}|x|_h^\alpha\left|\frac{1}{\Omega_Q^m} \int_{|(z_1,\ldots,z_m)|_h<1}f_1(\delta_{|x|_h}z_1)\cdots f_m(\delta_{|x|_h}z_m)dz_1\cdots dz_m\right|\\
		&\leq\prod_{i=1}^m\|f_i\|_{H^\infty_{\alpha_i}(\mathbb{H}^n)}\sup_{x\in\mathbb{H}^n\backslash\{0\}}|x|_h^\alpha\left|\frac{1}{\Omega_Q^m}\int_{|(y_1,\ldots,y_m)|_h<1}\prod_{i=1}^m|x|_h^{-\alpha}|y_i|_h^{-\alpha_i}dy_1\cdots dy_m\right|\\
		&=\mathcal{A}_m^h\prod_{i=1}^m\|f_i\|_{H^\infty_{\alpha_i}(\mathbb{H}^n)}.
	\end{aligned}
	$$
	This indicates for every $f_i\in\prod_{i=1}^m(H^\infty_{\alpha_i}),i=1,\ldots,m$, by taking the supremum in $H^\infty_{\alpha_i}(\mathbb{H}^n)$ that
	\begin{equation}\label{main_211}
		\|\mathfrak{H}_m^h\|_{\prod_{i=1}^mH^\infty_{\alpha_i}(\mathbb{H}^n)\rightarrow H^\infty_{\alpha}(\mathbb{H}^n)}\leq \mathcal{A}_m^h.
	\end{equation}
	Next, we take $f_{\alpha_j}(x)$ defined by (\ref{main_6}), which satisfies (\ref{main_8}).
	Using (\ref{main_21}), we can obtain
	\begin{equation}\label{main_222}
		|x|_h^\alpha\left|\frac{1}{\Omega_Q^m|x|_h^{mQ}}\int_{|(y_1,\ldots,y_m)|_h<|x|_h}f_{\alpha_1}(y_1)\cdots f_{\alpha_m}(y_m)dy_1\cdots dy_m\right|=\mathcal{A}_m^h.
	\end{equation}
	Combining (\ref{main_8}) with (\ref{main_222}), it follows that
	$$
	\|\mathfrak{H}_m^h\|_{\prod_{i=1}^mH^\infty_{\alpha_i}(\mathbb{H}^n)\rightarrow H^\infty_{\alpha}(\mathbb{H}^n)}\geq \mathcal{A}_m^h.
	$$
	Thus, we have finished the proof of Theorem \ref{mian_021}.
\end{proof}
In order to obtain the sharp constant, we learned and borrowed the method in \cite{11}.
For convenience, we take several special cases for separate calculations, and the general situation is a summary of these special cases.\\
\textbf{Case 1 when $m=1$.}

This situation returns to equation (\ref{main_331}), here we provide another form of definition, then its norm of sharp bound is
$$
\mathcal{A}^h_1=\frac{1}{\Omega_Q}\int_{|y|_h<1} |y|_h^{-\alpha} dy<\infty.
$$
Using the polar coordinate transformation formula, we have
$$
\begin{aligned}
	\mathcal{A}^h_1&=\frac{1}{\Omega_Q}\int_{|y|_h<1} |y|_h^{-\alpha} dy\\
	&=\frac{1}{\Omega_Q}\int_0^1\int_{S(0,1)} r^{-\alpha+Q-1}dr dy^{'}\\
	&=\frac{\omega_Q}{\Omega_Q}\int_0^1 r^{-\alpha+Q-1}d r\\
	&=\frac{Q}{Q-\alpha}.
\end{aligned}
$$
At this point, we have obtained the sharp constant when $m=1$, which is
\begin{equation}\label{case1}
	\|\mathfrak{H}_1^h\|_{H^\infty_\alpha(\mathbb{H}^n)}=\frac{Q}{Q-\alpha}.
\end{equation}\\
\textbf{Case 2 when $m=2$.}

According to  Theorem \ref{mian_021} and the polar coordinate transformation formula, we have
$$
\begin{aligned}
	\mathcal{A}^h_2&=\frac{1}{\Omega_Q^2}\int_{|(y_1,y_2)|_h<1}|y_1|_h^{-\alpha_1}|y_2|_h^{-\alpha_2}dy_1 d y_2\\
	&=\frac{1}{\Omega_Q^2}\int_0^1\int_0^1\int_{S(0,1)}\int_{S(0,1)}r_1^{-\alpha_1+Q-1}r_2^{-\alpha_2+Q-1}d r_1 d r_2 d y_1^{'}d y_2^{'}\\
	&=\frac{\omega^2_Q}{\Omega^2_Q}\frac{1}{2Q-\alpha}\int_0^1(1-t^2)^{\frac{1}{2}(-\alpha_1+Q-1)}t^{-\alpha_2+Q-1}(1-t^2)^{-\frac{1}{2}}d t\\
	&=\frac{\omega^2_Q}{\Omega^2_Q}\frac{1}{4Q-2\alpha}\int_0^1(1-x)^{\frac{1}{2}(-\alpha_1+Q)-1}x^{\frac{1}{2}(-\alpha_2+Q)-1} d x\\
	&=\frac{Q^2}{2(2Q-\alpha)}B(\frac{Q}{2}-\frac{\alpha_1}{2},\frac{Q}{2}-\frac{\alpha_2}{2})\\
	&=\frac{Q^2}{2(2Q-\alpha)}\frac{\Gamma(\frac{Q}{2}-\frac{\alpha_1}{2})\Gamma(\frac{Q}{2}-\frac{\alpha_2}{2})}{\Gamma(Q-\frac{\alpha}{2})}.
\end{aligned}
$$
This indicates that when $m=2$, the sharp constant is
\begin{equation}\label{case2}
	\mathcal{A}^h_2=\frac{Q^2}{2(2Q-\alpha)}\frac{\Gamma(\frac{Q}{2}-\frac{\alpha_1}{2})\Gamma(\frac{Q}{2}-\frac{\alpha_2}{2})}{\Gamma(Q-\frac{\alpha}{2})}
\end{equation}\\
\textbf{Case 3 when $m \geq 3$.}

The proof of case when $m\geq3$ is similar to that of case 2, we can continue to use the same method.
\begin{equation}\label{case3}
	\begin{aligned}
		\mathcal{A}^h_m=&\frac{1}{\Omega_Q^m}\int_{|(y_1,\ldots,y_m)|_h<1}\prod_{i=1}^m|y_i|_h^{-\alpha_i}dy_1\cdots d y_m\\
		=&\frac{1}{\Omega_Q^m}\int_0^1\cdots\int_0^1\int_{(S(0,1))
			^m}\prod_{i=1}^m r_i^{-\alpha_i+Q-1}d r_1\cdots d r_m d y_1^{'}\cdots dy_m^{'}\\
		=&\frac{\omega_Q^m}{\Omega_Q^m}\frac{2}{2^m(mQ-\alpha)}\int_0^1\cdots\int_0^1(1-x)^{\frac{1}{2}(-\alpha_1+Q)-1}x^{\frac{1}{2}(-\alpha_2+Q)-1}\\
		&\times\cdots\times(1-x)^{\frac{1}{2}(-\alpha_{m-1}+Q)-1}x^{\frac{1}{2}(-\alpha_m+Q)-1}dx\\
		=&\frac{2Q^m}{2^m(mQ-\alpha)}\frac{\prod_{i=1}^m\Gamma(\frac{Q}{2}-\frac{\alpha_i}{2})}{\Gamma(\frac{mQ}{2}-\frac{\alpha}{2})}.
	\end{aligned}
\end{equation}
When we take $m=1$ and $m=2$ in (\ref{case3}), the result is the same as (\ref{case1}) and (\ref{case2}).
So we obtained the sharp constant for $\mathcal{A}_m^h$, which is
$$
\mathcal{A}^h_m=\frac{2Q^m}{2^m(mQ-\alpha)}\frac{\prod_{i=1}^m\Gamma(\frac{Q}{2}-\frac{\alpha_i}{2})}{\Gamma(\frac{mQ}{2}-\frac{\alpha}{2})}.
$$

\section{Sharp bounds for  Hardy-Littlewood-P\'{o}lya and
	Hilbert operators}
In this section, we study $m$-linear $n$-dimensional Hardy-Littlewood-P\'{o}lya and Hilbert operators in weighted-type space on Heisenberg group. First of all, let's recall their definitions and properties.

The $m$-linear $n$-dimensional Hardy-Littlewood-P\'{o}lya is defined by
$$
P_m(f_1,\ldots , f_m)(x):=\int_{\mathbb{R}^{nm}}\frac{f_1(y_1)\cdots f_m(y_m)}{[\text{max}(|x|^n,|y_1|^n,\ldots, |y_m|^n)]^m}dy_1\cdots dy_m, x\in\mathbb{R}^n\backslash\{0\}.
$$
When $n=m=1$, the linear Hardy-Littlewood-P\'{o}lya operator $P_1$ is considered in \cite{12}, Hardy obtained that the norm of Hardy-Littlewood-P\'{o}lya operator on $L^p(\mathbb{R}^+)(1<q<\infty)$, that is
$$
\|P_1\|_{L^p(\mathbb{R}^+\rightarrow L^p(\mathbb{R}^+))}=\frac{p^2}{p-1}.
$$
The Hilbert operator is the essential extension of the classical Hilbert's inequality , the $m$-linear $n$-dimensional Hilbert operator is defined by
$$
P^*_m(f_1,\ldots, y_m)(x):=\int_{\mathbb{R}^{nm}}\frac{f_1(y_1)\cdots f_m(y_m)}{(|x|^n+|y_1|^n+\cdots+|y_m|^n)^m}d y_1\cdots dy_m,x\in\mathbb{R}^n\backslash\{0\}.
$$
For $n=1$, the following is a know sharp estimate
$$
\int_0^{\infty} T_1^* f(x) g(x) d x \leq \frac{\pi}{\sin (\pi / p)}\|f\|_{L^p(0, \infty)}\|g\|_{L^{p^{\prime}(0, \infty)}}.
$$
Now, we give the definitions of $m$-linear $n$-dimensional Hardy-Littlewood-P\'{o}lya and Hilbert operators on Heisenberg group.
\begin{defn}
	Suppose that $f_1,\ldots,f_m$ be  nonnegative locally integrable functions on $\mathbb{H}^n$. The $m$-linear $n$-dimensional Hardy-Littlewood-P\'{o}lya operator is defined by
	\begin{equation}
		P^h_m(f_1,\ldots,f_m)(x)=\int_{\mathbb{H}^{nm}}\frac{f_1(y_1)\cdots f_m(y_m)}{[\text{max} (|x|_h^Q,|y_1|_h^Q,\ldots,|y_m|_h^Q)]^m}dy_1\cdots dy_m,x\in\mathbb{H}^n\backslash\{0\}.
	\end{equation}
\end{defn}
\begin{defn}
	Suppose that $f_1,\ldots,f_m$ be nonnegative locally integrable functions on $\mathbb{H}^n$. The $m$-linear $n$-dimensional Hilbert operator is defined by
	\begin{equation}
		P^{h*}_m(f_1,\ldots,f_m)(x)=\int_{\mathbb{H}^{nm}}\frac{f_1(y_1)\cdots f_m(y_m)}{(|x|^Q_h+|y_1|_h^Q+\cdots+|y_m|_h^Q)^m}dy_1\cdots d y_m,x\in\mathbb{H}^n\backslash\{0\}.
	\end{equation}
\end{defn}
Now, we formulate our main results as follows.
\begin{thm}\label{thm1}
	Let $m\in\mathbb{N}, f_i\in H^\infty_{\alpha_i}(\mathbb{H}^n), \alpha>0$ and $\alpha=\alpha_1+\cdots+\alpha_m$ with $\alpha_i>0$ for any $i=1,\ldots,m$. If
	\begin{equation}
		\mathcal{B}^h_m=\int_{\mathbb{H}^{nm}}\frac{\prod_{i=1}^m |y_i|_h^{-\alpha_i}}{[\max(1,|y_1|_h^Q,\ldots,|y_m|_h^Q)]^m}dy_1\cdots dy_m<\infty,
	\end{equation}
	then $P^h_m$ is bounded from $H_{a_1}^\infty(\mathbb{H}^n)\times\cdots\times H_{\alpha_m}^\infty(\mathbb{H}^n)$ to $H_\alpha^\infty(\mathbb{H}^n)$ if and only if the following formula for the norm of the operator hold
	$$
	\|P_m^h\|_{\prod_{i=1}^m H^\infty_{\alpha_i}(\mathbb{H}^n)\rightarrow H^\infty_\alpha(\mathbb{H}^n)}=\frac{mQ\omega_Q^m}{\alpha\prod_{j=1}^m(Q-\alpha_j)}.
	$$
\end{thm}
\begin{thm}\label{thm2}
	Let $m\in\mathbb{N}$,$f_i\in H^\infty_{\alpha_i}(\mathbb{H}^n)$,$\alpha>0$ and $\alpha=\alpha_1+\cdots+\alpha_m$ with $\alpha_i>0$ for any $i=1,\ldots,m$. If
	\begin{equation}
		\mathcal{B}^{h*}_m=\int_{\mathbb{H}^{nm}}\frac{\prod_{i=1}^m |y_i|_h^{-\alpha_i}}{[ \max(1+|y_1|_h^Q+\cdots+|y_m|_h^Q)]^m}dy_1\cdots dy_m<\infty,
	\end{equation}
	then $P^h_m$ is bounded from $H_{a_1}^\infty(\mathbb{H}^n)\times\cdots\times H_{\alpha_m}^\infty(\mathbb{H}^n)$ to $H_\alpha^\infty(\mathbb{H}^n)$ if and only if the following formula for the norm of the operator hold
	$$
	\|P_m^{h*}\|_{\prod_{i=1}^m H^\infty_{\alpha_i}(\mathbb{H}^n)\rightarrow H^\infty_\alpha(\mathbb{H}^n)}=\left(\frac{2 \pi^{n+\frac{1}{2}} \Gamma(n / 2)}{(n+1) \Gamma(n) \Gamma((n+1)/2)}\right)^m\frac{\prod^m_{i=1}\Gamma(1-\frac{\alpha_i}{Q})\Gamma(\frac{\alpha}{Q})}{\Gamma(m)}.
	$$
\end{thm}
In the above proof,we have obtained that $\mathcal{B}^h_m$ and $\mathcal{B}^{h*}_m$ are clearly valid, so we omit their proof and will focus on proving  their sharp constants.
\begin{proof}[Proof of Theorem \ref{thm1} and \ref{thm2}]
	In order to get the sharp constant of $m$-linear $n$-dimensional Hardy-Littlewood-P\'{o}lya operator, we will consider the following two scenarios and combine them to obtain our result.
	
	\textbf{Case 1 when $m=2$.}
	
	In this case , we have
	$$
	\mathcal{B}^h_m=\int_{\mathbb{H}^n}\int_{\mathbb{H}^n}\frac{|y_1|_h^{-\alpha_1}|y_2|_h^{\alpha_2}}{[\text{max}(1,|y_1|_h^Q,|y_2|_h^Q)]^2}dy_1dy_2.
	$$
	By calculation,we have
	$$
	\begin{aligned}
		\int_{\mathbb{H}^n}\int_{\mathbb{H}^n}\frac{|y_1|_h^{-\alpha_1}|y_2|_h^{\alpha_2}}{[\text{max}(1,|y_1|_h^Q,|y_2|_h^Q)]^2}dy_1dy_2
		=&\int_{|y_1|_h<1}\int_{|y_2|_h<1}|y_1|_h^{-\alpha_1}|y_2|_h^{-\alpha_2}dy_1dy_2\\
		&+\int_{|y_1|_h>1}\int_{|y_2|_h\leq|y_1|_h}|y_1|_h^{-\alpha_1-2Q}|y_2|_h^{-\alpha_2}dy_1 dy_2\\
		&+\int_{|y_2|_h>1}\int_{|y_1|_h<|y_2|_h}|y_1|_h^{-\alpha_1}|y_2|_h^{-\alpha_2-2Q}dy_1dy_2\\
		:=&I_0+I_1+I_2.
	\end{aligned}
	$$
	$$
	\begin{aligned}
		I_0=&\int_{|y_1|_h<1}\int_{|y_2|_h<1}|y_1|_h^{-\alpha_1}|y_2|_h^{-\alpha_2}dy_1dy_2\\
		=&\frac{\omega_Q^2}{(Q-\alpha_1)(Q-\alpha_2)},\\
		I_1=&\int_{|y_1|_h>1}\int_{|y_2|_h\leq|y_1|_h}|y_1|_h^{-\alpha_1-2Q}|y_2|_h^{-\alpha_2}dy_1 dy_2\\
		=&\frac{\omega_Q}{Q-\alpha_2}\int_{|y_1|_h>1}|y_1|_h^{-\alpha-Q}dy_1\\
		=&\frac{\omega_Q^2}{\alpha(Q-\alpha_2)}.
	\end{aligned}
	$$
	Similarly,
	$$
	\begin{aligned}
		I_2&=\int_{|y_2|_h>1}\int_{|y_1|_h<|y_2|_h}|y_1|_h^{-\alpha_1}|y_2|_h^{-\alpha_2-2Q}dy_1dy_2\\
		&=\frac{\omega_Q}{Q-\alpha_1}\int_{|y_2|_h>1}|y_2|_h^{-\alpha-Q}dy_2\\
		&=\frac{\omega_Q^2}{\alpha(Q-\alpha_1)}.
	\end{aligned}
	$$
	Thus,we have
	$$
	\begin{aligned}
		&\int_{\mathbb{H}^n}\int_{\mathbb{H}^n}\frac{|y_1|_h^{-\alpha_1}|y_2|_h^{\alpha_2}}{[\text{max}(1,|y_1|_h^Q,|y_2|_h^Q)]^2}dy_1dy_2\\
		=&I_0+I_1+I_2\\
		=&\frac{2Q\omega_Q^2}{\alpha(Q-\alpha_1)(q-\alpha_2)}.
	\end{aligned}\\
	$$
	
	\textbf{Case 2 when $m\geq3$.}
	
	Let
	$$
	\begin{aligned}
		&E_0=\{(y_1,\ldots,y_m)\in\mathbb{H}^n\times\cdots\times\mathbb{H}^n:|y_k|_h\leq1,1\leq k\leq m\};\\
		&E_1=\{(y_1,\ldots,y_m)\in\mathbb{H}^n\times\cdots\times\mathbb{H}^n:|y_1|_h>1,|y_k|_h\leq|y_1|_h,2\leq k\leq m\};\\
		&E_i=\{(y_1, \ldots, y_m) \in \mathbb{H}^n \times \cdots \times \mathbb{H}^n:|y_i|_h>1,|y_j|_h<|y_k|_h,|y_k|_h \leq|y_i|_h, 1 \leq j<i<k \leq m\};\\
		&E_m=\{(y_1,\ldots,y_m)\in\mathbb{H}^n\times\cdots\times\mathbb{H}^n:|y_m|_h>1,|y_j|_h<|y_m|_h,1<j<m\}.
	\end{aligned}
	$$
	Obviously, we have
	$$
	\bigcup_{j=0}^m E_j=\mathbb{H}^n \times \cdots \times \mathbb{H}^n,E_i \cap E_j=\emptyset.
	$$
	Taking
	$$
	K_j=\int_{\mathbb{H}^{nm}}\frac{\prod_{i=1}^m |y_i|_h^{-\alpha_i}}{[ \text{max}  (1,|y_1|_h^Q,\ldots,|y_m|_h^Q)]^m}dy_1\cdots dy_m ,
	$$
	then we begin to calculate $K_j$ with $j=0,1,\ldots,m$.
	$$
	K_0=\prod_{j=1}^m\int_{|y_j|\leq 1}|y_j|_h^{-\alpha_j}d y_j=\frac{\omega_Q^m}{\prod_{j=1}^m(Q-\alpha_j)},
	$$
	$$
	\begin{aligned}
		K_1&=\int_{|y_1|>1}|y_1|_h^{-\alpha_1-Qm}dy_1\prod_{j=2}^m\int_{|y_j|_h\leq|y_1|_h}|y_j|_h^{-\alpha_j}dy_j\\
		&=\frac{\omega_Q^{m-1}}{\prod_{j=2}^m(Q-\alpha_j)}\int_{|y_1|_h>1}|y_1|_h^{-\alpha-Q}dy_1\\
		&=\frac{\omega_Q^m}{\alpha\prod_{j=2}^m(Q-\alpha_j)}.
	\end{aligned}
	$$
	So we can deduce that
	$$
	K_j=\frac{\omega_Q^m}{\alpha \prod_{1\leq i\leq m,i\neq j} (Q-\alpha_j)}.
	$$
	Then we can obtain that
	$$
	K_m=\frac{mQ\omega_Q^m}{\alpha\prod_{j=1}^m(Q-\alpha_j)}.
	$$
	Combining the above two cases, we have completed the Theorem \ref{thm1}.
	Next, we follow the method given by B\'{e}nyi and Oh \cite{13} to calculate the sharp constant of $m$-linear $n$-dimensional Hilbert operator. Using the polar coordinates and making change of variables, we have
	$$
	\begin{aligned}
		\|P_m^{h*}\|_{\prod_{i=1}^m H^\infty_{\alpha_i}(\mathbb{H}^n)\rightarrow H^\infty_\alpha(\mathbb{H}^n)}&=\int_{\mathbb{H}^{nm}}\frac{|y_1|_h^{-\alpha_1}\cdots|y_m|_h^{-\alpha_m}}{(1+|y_1|_h^Q+\cdots+|y_m|_h^Q)} dy_1\cdots dy_m\\
		&=\omega_Q^m\int_0^\infty\cdots\int_0^\infty\frac{r_1^{Q-\alpha_1-1}\cdots r_m^{Q-\alpha_m-1}}{(1+r_1^Q+\cdots+r_m^Q)^m}dr_1\cdots dr_m\\
		&=\frac{\omega_Q^m}{Q^m}\int_0^\infty\cdots\int_0^\infty\frac{t_1^{\frac{1}{Q}(Q-\alpha_1)-1}\cdots t_m^{\frac{1}{Q}(Q-\alpha_m)-1}}{(1+t_1+\cdots+t_m)^m}d t_1\cdots d t_m\\
		&=\frac{\omega_Q^m}{Q^m}\int_0^\infty\cdots\int_0^\infty\frac{t_1^{-\frac{\alpha_1}{Q}}\cdots t_m^{-\frac{\alpha_m}{Q}}}{(1+t_1+\cdots+t_m)^m}d t_1\cdots d t_m\\
		&=\frac{\omega_Q^m}{Q^m}\int_0^\infty\cdots\int_0^\infty\frac{t_1^{-\beta_1}\cdots t_m^{-\beta_m}}{(1+t_1+\cdots+t_m)^\alpha}dt_1\cdots dt_m.
	\end{aligned}
	$$
	Set
	$$
	I_m(\alpha,\beta_1,\ldots,\beta_m):=\int_0^\infty\cdots \int_0^\infty\frac{t_1^{-\beta_1}\cdots t_m^{-\beta_m}}{(1+t_1+\cdots+t_m)^\alpha}dt_1\cdots dt_m.
	$$
	By a simple calculate, we have
	$$
	\begin{aligned}
		\int_0^\infty\frac{1}{(1+t)^\alpha t^\beta}dt&=\int_0^\infty(1-t)^{-\beta}t^{\alpha+\beta-2}dt\\
		&=B(1-\beta,\alpha+\beta-1).
	\end{aligned}
	$$
	Using the variable substitution $t_m=(1+t_1+\cdots+t_m-1)q_m$, we have
	$$
	\begin{aligned}
		&I_m(\alpha,\beta_1,\ldots,\beta_m)\\
		=&\int_0^\infty\cdots\int_0^\infty\frac{t_1^{-\beta_1}\cdots t_{m-1}^{-\beta_{m-1}}}{(1+t_1+\cdots+t_{m-1})^{\alpha-1+\beta_m}}dt_1\cdots dt_{m-1}\int_0^\infty\frac{1}{(1+q_m)^\alpha q_m^{\beta_m}}d q_m\\
		=&B(1-\beta_m,\alpha+\beta_m-1)I_m(\alpha-1+\beta_m,\beta_1,\ldots,\beta_{m-1}).
	\end{aligned}
	$$
	Using the inductive method combine with the properties of Gamma function, we have
	$$
	I_m(\alpha,\beta_1,\ldots,\beta_m)=\frac{\prod_{i=1}^m \Gamma(1-\beta_i) \Gamma(\alpha-k+\prod_{i=1}^m \beta_i)}{\Gamma(\alpha)}.
	$$
	Thus, we can obtain
	$$
	\begin{aligned}
		\|P_m^{h*}\|_{\prod_{i=1}^m H^\infty_{\alpha_i}(\mathbb{H}^n)\rightarrow H^\infty_\alpha(\mathbb{H}^n)}&=\frac{\omega_Q^m}{Q^m}\frac{\prod^m_{i=1}\Gamma(1-\frac{\alpha_i}{Q})\Gamma(\frac{\alpha}{Q})}{\Gamma(m)}\\
		&=\left(\frac{2 \pi^{n+\frac{1}{2}} \Gamma(n / 2)}{(n+1) \Gamma(n) \Gamma((n+1)/2)}\right)^m\frac{\prod^m_{i=1}\Gamma(1-\frac{\alpha_i}{Q})\Gamma(\frac{\alpha}{Q})}{\Gamma(m)}.
	\end{aligned}
	$$
	This finishes the proof of Theorem \ref{thm1} and \ref{thm2}.
\end{proof}

\subsection*{Acknowledgements}

This work was supported by National Natural Science Foundation of  China (Grant No. 12271232) and Shandong Jianzhu University Foundation (Grant No. X20075Z0101).

\begin{flushleft}

\vspace{0.3cm}\textsc{Xiang Li\\School of Science\\Shandong Jianzhu University\\Jinan, 250000\\P. R. China}

\emph{E-mail address}: \textsf{lixiang162@mails.ucas.ac.cn}

\vspace{0.3cm}\textsc{Zhanpeng Gu\\School of Science\\Shandong Jianzhu University \\Jinan, 250000\\P. R. China}

\emph{E-mail address}: \textsf{guzhanpeng456@163.com}

\vspace{0.3cm}\textsc{Dunyan Yan\\School of Mathematical Sciences\\University of Chinese Academy of Sciences\\Beijing, 100049\\P. R. China}

\emph{E-mail address}: \textsf{ydunyan@ucas.ac.cn}

\vspace{0.3cm}\textsc{Zhongci Hang\\School of Science\\Shandong Jianzhu University \\Jinan, 250000\\P. R. China}

\emph{E-mail address}: \textsf{babysbreath4fc4@163.com}

\end{flushleft}


\begin{thebibliography}{99}
	
	


\bibitem{3} W.Beckner  Pitt's inequality with sharp convolution estimates. Proc Amer Math Soc, 2008, 136: 1871--1885.

\bibitem{4} W.Beckner  Weighted inequalities and Stein-Weiss potentials. Forum Math, 2008, 20: 587--606.

\bibitem{13} \'{A}. B\'{e}nyi and T. Oh, Best constants for certain multilinear integral operators. J. Inequal.Appl,2006, Article ID 28582, 1--12.

\bibitem{7} G.B. Folland and E.M. Stein, Hardy spaces on homogeneous groups, Princeton, N. J. Princeton University Press, 1982.


\bibitem{11}Z. W. Fu, L. Grafakos, S. Z. Lu, and F. Y. Zhao, Sharp bounds for $m$-linear Hardy and Hilbert operators, Houston Journal of Mathematics, 2012, 38(1): 225--244.

\bibitem{12}G.H. Hardy, J.E. Littlewood and G. P\'{o}lya, Inequalities, 2nd ed., Cambridge University Press, Cambridge, UK, 1952.

\bibitem{6}A. Kor\"{a}anyi and H.M. Reimann, Quasiconformal mappings on the Heisenberg group, Invent. Math., 80 ,1985, 309--338.

\bibitem{5}Coulhon T, Muller D, Zienkiewicz J. About Riesz transforms on the Heisenberg groups. Math
Ann, 1996, 305(2): 369--379.

\bibitem{1} S. Stevi'c, Note on norm of an $m$-linear integral-type operator between weighted-type spaces. Adv. Differ. Equ. 2021, 187--198.

\bibitem{2} E. M. Stein and G. Weiss, Fractional integrals on $n$-dimensional Euclidean space, J. Math. Mech. 7 ,1958, 503--514.



\bibitem{8} S. Thangavelu, Harmonic analysis on the Heisenberg group, Progress in Mathematics, vol. 159, Boston, MA: Birkh¨auser Boston, 1998.

\bibitem{9}D. Wu and D.Y. Yan, Sharp constants for a class of multilinear integral operators and some applications, Sci. China Math., 59 (5) (2016), 907--920.

\bibitem{10}Q.Y. Wu and Z.W.Fu ,Sharp estimates for Hardy operators on Heisenberg group,Front.Math.China 2016,11(1):155--172.





\end{thebibliography}
\end{document}